\documentclass[12pt]{article}

\usepackage{amsmath,amsthm,amsfonts,amssymb,amscd,caption,color,subcaption,cite}
\usepackage{graphicx}
\usepackage[mathscr]{eucal}
\allowdisplaybreaks[4]
\usepackage[figurename=Fig.]{caption}
\numberwithin{equation}{section}
\allowdisplaybreaks[4]

\newtheorem {Lemma}{Lemma}[section]
\newtheorem {Theorem} {Theorem}[section]

\newtheorem {Claim} {Claim}[section]
\usepackage{tikz}
\usetikzlibrary{positioning}
\usepackage{circuitikz}
\usetikzlibrary{patterns}
\usepackage{fullpage}

\begin{document}

\title{Unbalanced signed bipartite graphs containing no negative $C_4$ with maximum spectral radius}

\author{Yiting Cai$^{a}$\footnote{E-mail: yitingcai@m.scnu.edu.cn}, Hongying Lin$^{b}$\footnote{E-mail: linhy99@scut.edu.cn}, Bo Zhou$^{a}$\footnote{E-mail: zhoubo@scnu.edu.cn} \\
$^{a}$School of  Mathematical Sciences, South China Normal University,\\
Guangzhou 510631, P.R. China\\
$^{b}$ School of Mathematics, South China University of Technology,\\ 
Guangzhou 510641, P.R. China
}

\date{}
\maketitle

\begin{abstract}
A signed graph $(G,\sigma)$ is a graph $G$ together with an assignment $\sigma$ of either a positive sign or a negative sign to each edge.
A signed graph is unbalanced if it contains a cycle with odd number of negative edges. The spectral radius of a signed graph is the spectral radius of its adjacency matrix, in which for vertices $u,v$, the  $(u,v)$-entry is $0$, $-1$, or $1$ depending on whether
$uv$ represents no edge, a negative edge, or a positive edge, respectively. Recently,  Conde, Dratman and Grippo [Discrete Math. 349 (2026) 114942]  proved that there is only one unbalanced signed bipartite graph with maximum spectral radius, up to switching isomorphism. In this paper, we establish a spectral Tur\'an type results for signed bipartite graphs. More precisely, we determine the unique graphs  containing no negative cycles of length four with maximum spectral radius, up to switching isomorphism, among  unbalanced signed bipartite graphs with fixed bipartite sizes and order, respectively. 
\end{abstract}

\section{Introduction}
We consider simple graphs. Denote by $C_m$ a cycle of order $m\ge 3$. Denote by $K_n$ a complete graph of order $n$.
Let $G$ be a  bipartite graph and $X$ and $Y$  its partite sets. If $|X| =r$ and $|Y| = s$,  then we say that $G$ has bipartite sizes $r$ and $s$ or $G$ is an $r$ by $s$ bipartite graph.  Denote by $K_{r,s}$ a complete $r$ by $s$ bipartite graph. 
There are studies on the  maximum number of edges in $r$ by $s$ bipartite graphs or of  bipartite graphs with given order  without certain cycle(s). 
For example, de Caen and Sz\'{e}kely \cite{deZ} studied the maximum number of edges in  $r$ by $s$ bipartite graphs without $C_4$ or a $C_6$, Hoory \cite{Ho} derived upper bounds for the sizes of $r$ by $s$ bipartite graphs with given girth, and   Naor and  Verstera\" ete \cite{NV} gave 
 upper bounds for the sizes of $r$ by $s$ bipartite graphs without  $C_{2k}$. Such pronlems belong to  the classical Tur\'{a}n-type extremal problem that is to determine 
the  maximum number of edgesin  graphs  without certain graph(s). 

A signed graph $\Gamma=(G, \sigma)$ is a simple graph $G$ (called the underlying graph) with vertex set $V(G)$ and edge set $E(G)$ together with a function $\sigma$ defined from  $E(G)$ into $\{-1,1\}$ (called the signature). The adjacency matrix $A(\Gamma)=(a_{uv})_{u,v\in V(G)}$ of $\Gamma$ is defined as
\[
a_{uv}=\begin{cases} \sigma(uv) & \mbox{if $uv\in E(G)$},\\
0 & \mbox{otherwise}.
\end{cases}
\]
The spectrum of $A(\Gamma)$ is called the spectrum of the signed graph $\Gamma$.
For a signed graph $\Gamma$ of order $n$ with $m$ edges,
its eigenvalues $\lambda_1(\Gamma)\ge \lambda_2(\Gamma)\ge \dots\ge\lambda_n(\Gamma)$   are all real numbers because $A(\Gamma)$ is a real symmetric matrix. 
It may be easily checked that $\sum_{i=1}^n\lambda_i^2(\Gamma)=2m$. 
The largest eigenvalue $\lambda_1(\Gamma)$ is often called the index. The spectral radius $\rho(\Gamma)$ of $\Gamma$ is the largest absolute of its eigenvalues, i.e.,
$\rho(\Gamma)=\max\{\lambda_1(\Gamma), -\lambda_n(\Gamma)\}$.


If $U$ is any subset of the vertices of a signed graph $\Gamma$, then switching in $U$ reverses the signs of all edges between vertices $u$ and $v$ for which $u\in U$ and $v\not\in U$. The resulted signed graph is denoted by $\Gamma (U)$.
Signed graphs are switching equivalent if one can be obtained from the other by switching in some subset $U$. Obviously, a signed graph is switching equivalent to itself.
Two signed graphs are switching isomorphic if one is isomorphic to a signed graph that is switching equivalent to the other.

For a signed graph $\Gamma=(G, \sigma)$ with $e\in E(G)$,
if $\sigma(e)=1$ ($\sigma(e)=-1$, respectively), then the edge $e$ is positive (negative, respectively).
%
If each edge of a signed graph $(G,\sigma)$ is positive, then it is denoted by $(G,+)$. An ordinary graph $G$ is also viewed as $(G,+)$.

A cycle  in $\Gamma$ is  positive (negative, respectively) if the number of its negative edges is even (odd, respectively). We use $C_m^-$ to denote any negative cycle of order $m\ge 3$.
A signed graph is called balanced (unbalanced, respectively) if it contains no (at least one, respectively) negative cycle, see \cite{Z2, Za}.
It follows by \cite[Lemma 5.3]{Z2} that a signed graph $(G,\sigma)$ is balanced if and
only if it is switching equivalent to $(G,+)$.




The index of unbalanced signed graphs received due attention, see, e.g.  \cite{ABH,BS,BS2,HLSW,WL,WPH}. For example,
Wang and Lin \cite{WL} determined the signed graphs with maximum index among unbalanced  signed graphs containing no $C_4^-$.

The spectral radius of unbalanced signed graphs was also studied, see, e.g.  \cite{BBC,BS2,BT,CY, WHL, W, CZ}. IFor example the signed graphs with maximum spectral radius have been determined among unbalanced  signed graphs containing no negative $K_r$ for $r=3,4,5$ \cite{CY, WHL, W}. The signed graphs with maximum spectral radius have been determined among unbalanced  signed graphs containing no negative $C_3$ and $C_4$ \cite{CZ}.

For a signed graph $\Gamma$, 
let $-\Gamma$ be the signed graph obtained from $\Gamma$ by reversing the sign of each edge.
If $\Gamma$ and $-\Gamma$ are switching equivalent, then we say $\Gamma$ is sign-symmetric, and in this case,
 $\rho(\Gamma)=\lambda_1(\Gamma)=-\lambda_n(\Gamma)$.
For a signed bipartite graph  $\Gamma=(G,\sigma)$, $G$ is bipartite, so it is  sign-symmetric, that is, 
 $\rho(\Gamma)=\lambda_1(\Gamma)$ \cite{BB}.
For more results about sign-symmetric signed graphs, see \cite{BC}.
Very recently,  Conde,  Dratman and Grippo \cite{CDG} proved that among connected unbalanced signed bipartite graphs of order $n\ge 2$, a signed graph maximizes the spectral radius if and only if it is switching isomorphic to
the signed graph $(K_{\lfloor\frac{n}{2}\rfloor, \lceil\frac{n}{2}\rceil}, \sigma)$ with exactly one negative edge.

%

Note that there are negative $C_4$s in the signed graph $(K_{\lfloor\frac{n}{2}\rfloor, \lceil\frac{n}{2}\rceil}, \sigma)$ with exactly one negative edge.  So it is of interest to determine the signed graphs maximzing the spectral radius among unbalanced signed $r$ by $S$ bipartite graphs containing no  $C_4^{-}$, and among unbalanced signed  bipartite graphs of fixed order containing no  $C_4^{-}$, respectively. 

Let $r$ and $s$ be two integers with $3\le r\le s$. 
Let $\Gamma_{r,s}$  be the signed bipartite graph obtained from a copy of $(K_{r-1, s-1},+)$ by deleting an edge, say $uv$, and adding a path of length three to connecting $u$ and $v$ whose central edge is the only negative edge, see Fig. \ref{f2}. 


\begin{figure}[htbp]
	\centering
	%
	%
	\begin{tikzpicture}
		\filldraw [black] (1,-0.8) circle (2pt);
		\filldraw [black] (2.2,-0.8) circle (2pt);
		\filldraw [black] (3.4,-0.8) circle (2pt);
		\filldraw [black] (3.8,-0.8) circle (0.8pt);
		\filldraw [black] (4.1,-0.8) circle (0.8pt);
		\filldraw [black] (4.4,-0.8) circle (0.8pt);
		\filldraw [black] (4.8,-0.8) circle (2pt);
		\filldraw [black] (1,-3) circle (2pt);
		\filldraw [black] (2.2,-3) circle (2pt);
		\filldraw [black] (3.4,-3) circle (2pt);
		\filldraw [black] (3.8,-3) circle (0.8pt);
		\filldraw [black] (4.1,-3) circle (0.8pt);
		\filldraw [black] (4.4,-3) circle (0.8pt);
		\filldraw [black] (4.8,-3) circle (2pt);
		\draw (3.5,-3) ellipse (2.1 and 0.3);
		\draw (3.5,-0.8) ellipse (2.1 and 0.3);
		
		\draw  [red,line width=1.8pt] (1,-0.8)--(1,-3);
		\node at (0.7,-1.9) {$-$};
		\draw  [black](1,-0.8)--(2.2,-3)--(3.4,-0.8);
		\draw  [black](1,-3)--(2.2,-0.8);
		\draw  [black](3.4,-3)--(2.2,-0.8);
		\draw  [black](4.8,-3)--(2.2,-0.8);
		\draw  [black](3.4,-3)--(3.4,-0.8);
		\draw  [black](4.8,-3)--(3.4,-0.8);
		\draw  [black](2.2,-3)--(4.8,-0.8);
		\draw  [black](3.4,-3)--(4.8,-0.8);
		\draw  [black](4.8,-3)--(4.8,-0.8);
		\node at (3.5,-0.3) {$r-1$};
		\node at (3.5,-3.5) {$s-1$};
		\node at (2.8, -4.1) {$\Gamma_{r,s}$};
		%
	\end{tikzpicture}

	\caption{Signed graphs $\Gamma_{r,s}$.}
	\label{f2}
\end{figure}
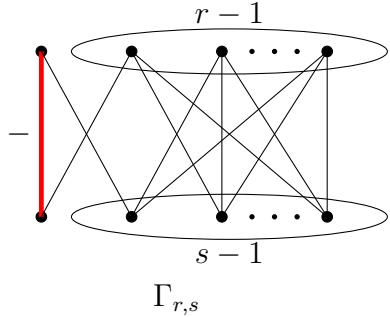

Let  $\Gamma=(G, \sigma)$ be an unbalanced signed graph. Then any negative cycle with minimum length is an induced cycle. Otherwise, suppose that $C$ be a negative cycle with minimum length but it is not induced. Then there is a chord, say  $e=uv$, 
on  $C$, so $u$ and $v$ divide $C$ into two paths, say $P$ and $Q$ connecting $u$ and $v$. As $C$ is negative,  the parity of the number of negative edges on $P$ and $Q$ are different, say $P$ contains an odd number of negative edges and $Q$ contains an even number of negative edges. If $e$ is positive, $P+e$ is a negative cycle, and if $e$ is negative, then $Q+e$ is a negative cycle, so we have a shorter negative cycle in $\Gamma$, a contradiction. Now 
suppose that $\Gamma$ is an unbalanced signed  bipartite graph with partite sizes $r$ and $s$ containing no $C_4^-$, where $r\le s$. Then it contains a negative induced cycle of length at least $6$, so it is necessary that $r\ge 3$.  

In this paper, we prove the following spectral Tur\'an type results for unbalanced signed bipartite graphs. There seems no such results for signed bipartite graphs in the literature at present.

\begin{Theorem}\label{N1}
	Let $\Gamma$ be an unbalanced signed $r$ by $s$ bipartite graph containing no $C_4^-$, where $3\le r\le s$.
	Then
	\[
	\rho(\Gamma)\le \sqrt{ \frac{1}{2}\left((r-1)(s-1)+2+\sqrt{((r-1)(s-1)+2)^2-4(2s-3)(2r-3)}\right)}, 
	\]
and equality  holds if and only if $\Gamma$ is switching isomorphic to $\Gamma_{r,s}$.
\end{Theorem}

\begin{Theorem}\label{N2}
Let $\Gamma$ be an  unbalanced signed bipartite graph of order $n\ge 6$ containing no $C_4^-$.
Then
\[
\rho(\Gamma)\le \begin{cases}
	\frac{1}{4}\left(n-6+\sqrt{(n-2)(n+6)}\right) & \text{if } n \text{ is even}, \\
	\sqrt{\frac{1}{8}\left(n^2-4n+11+\sqrt{(n^2-4n+11)^2-64(n-2)(n-4)}\right)} & \text{otherwise},
\end{cases}
\]and equality  holds if and only if $\Gamma$ is switching isomorphic to $\Gamma_{\left\lfloor\frac{n}{2}\right\rfloor, \left\lceil\frac{n}{2}\right\rceil}$.
\end{Theorem}

\section{Preliminaries}

For a signed graph $\Gamma=(G,\sigma)$ with $U\subset V(G)$,  $\Gamma-U$ denotes the signed graph $(G-U,\sigma_1)$ with $\sigma_1(e)=\sigma(e)$ for any $e\in E(G-U)$, where $G-U$ is the graph obtained from $G$ by removing the vertices in $U$ (and every edge incident to some vertex of $S$).
For a vertex $u$ in $\Gamma$, we denote by $N_\Gamma(u)$ the neighborhood of $u$ in $\Gamma$.

For a signed graph $\Gamma=(G,\sigma)$ of order $n$ with $E'\subset E(G)$, $\Gamma-E'$ denotes the signed graph $(G-E', \sigma')$ with $\sigma'(e)=\sigma(e)$ for any $e\in E(G)\setminus E'$, where $G-E'$  is the graph obtained from $G$ by removing the edges in $E'$.
If $E'=\{uv\}$, then we write $\Gamma-uv$ for $\Gamma-\{uv\}$.
If $E''\subset E(K_n)\setminus E(G)$, then $\Gamma+E''=(G+E'', \sigma'')$ with $\sigma''(e)=\sigma(e)$ for any $e\in E(G)$ and  $\sigma''(e)=1$ for any $e\in E''$,  where $G+E''$  is the graph obtained from $G$ by adding  the edges in $E''$.
If $E''=\{uv\}$, then we write $\Gamma+uv$ for $\Gamma+\{uv\}$.

Let $B$ be an $n\times n$ matrix whose rows and columns are indexed by elements in $X=[n]:=\{1, \dots, n\}$.  Let $\pi=\{X_1, \dots, X_t\}$ be a partition of $X$. For $1\le i,j\le t$, let $B_{ij}$ be the submatrix of $B$ whose rows and columns are indexed by elements of $X_i$ and $X_j$. The partition $\pi$ is equitable if the row sum of each $B_{ij}$ is a constant for $1\le i,j\le t$. The $t\times t$ matrix whose $(i,j)$-entry is the average row sum of $B_{ij}$ with $1\le i,j\le t$ is called the quotient matrix of $B$.

\begin{Lemma}\label{QM}\cite[Lemma 2.3.1]{BH}
Let $B$ be a real symmetric matrix. Then the spectrum of the quotient matrix of $B$ with respect to an equitable partition is contained in the spectrum of $B$.
\end{Lemma}

\begin{Lemma}\label{under}\cite{St1}
For a signed graph $\Gamma=(G,\sigma)$, we have $\lambda_1(\Gamma)\le\lambda_1(G)$ with equality when it is connected if and only if $\Gamma$ is balanced.
\end{Lemma}

\begin{Lemma}\label{SE}\cite{Z1}
Two signed graphs with the same underlying graph are switching equivalent if and only if they have the same class of positive cycles.
\end{Lemma}

\begin{Lemma}\label{same}\cite{Za}
Switching equivalent signed graphs have the same  spectrum.
\end{Lemma}


\begin{Lemma}\label{L+}\cite{Stan,SL}
A signed graph $\Gamma$ is switching equivalent to a signed graph $\Gamma^*$ such that $\lambda_1(\Gamma^*)$ has a non-negative eigenvector.
\end{Lemma}

Let $\Gamma$ be a signed graph on vertices $v_1,\dots, v_n$.  A real vector $\mathbf{x}=(x_1, \dots, x_n)^\top$ is viewed as a function on $\{v_1, \dots, v_n\}$ that maps vertex $v_i$  to $x_i$ , i.e.,
$\mathbf{x}(v_i)=x_i$, for $i=1,\dots, n$.

\begin{Lemma} \cite{BS} \label{add}  Let $\Gamma$ be a signed graph on vertices $v_1,\dots, v_n$ and  $\mathbf{x}=(x_1, \dots, x_n)^\top$  be an eigenvector associated
with the index of $\Gamma$. For vertices $v_k, v_\ell$, if $x_kx_\ell\ge 0$, at least one of $x_k,x_\ell$ is nonzero, and $v_k$ and $v_\ell$ are not adjacent ($v_kv_\ell$  is a negative edge, respectively), then for $\Gamma'=\Gamma+v_kv_\ell$ ($\Gamma'=\Gamma-v_kv_\ell$ or
$\Gamma'$ is  obtained from $\Gamma$ by reversing the sign of $v_kv_\ell$, respectively), we have $\lambda_1(\Gamma')>\lambda_1(\Gamma)$.
\end{Lemma}

%
%
%

\begin{Lemma}\label{comp}
Let  $r$ and $s$ be two integers with $3\le r\le s$ and $s\ge 4$, and let
\[
f(r,s)=\sqrt{ \frac{1}{2}\left((r-1)(s-1)+2+\sqrt{((r-1)(s-1)+2)^2-4(2r-3)(2s-3)}\right)}.
\]
Then 
\begin{enumerate}
\item[(i)]  $\lambda_1(\Gamma_{r, s})=f(r,s)$.

\item[(ii)] $f(r,s)>\sqrt{(r-1)(s-2)}$.

\item[(iii)] $s=n-r$ and $n$ is fixed,  
then $f(r,s)$ is strictly increasing for $r\le s$.
\end{enumerate}
\end{Lemma}

\begin{proof} Let $V_1, V_2$ be the bipartite sets of $K_{r-1, s-1}$  with  $|V_1|=r-1$ and $|V_2|=s-1$. 
Recall that  $\Gamma_{r,s}$  is  obtainable from a copy of $(K_{r-1, s-1},+)$  by deleting an edge $uv$  with $u\in V_1$ and $v\in V_2$ and adding a path of length three, say $uv_1u_1v$ 
to connect $u$ and $v$, where $u_1, v_1$ lie outside  $(K_{r-1, s-1},+)$ and $u_1v_1$  is the only negative edge. 
Evidently,  $r-2$ rows and $s-2$ rows of $A(\Gamma_{r, s})$ corresponding  to the vertices in $V_1\setminus\{u\}$ and in $V_2\setminus\{v\}$ respectively are equal. So the rank of $A(\Gamma_{r, s})$ is at most $6$, that is,  $0$ is an eigenvalue of $\Gamma_{r,s}$  with multiplicity at least $n-6$.

On the other hand, we partition $V(\Gamma_{r, s})$ as $\{u_1\}\cup \{u\}\cup (V_1\setminus\{u\})\cup
\{v_1\}\cup \{v\}\cup (V_2\setminus\{v\}\})$. Then  $A(\Gamma_{r, s})$ can be partitioned correspondingly. It is easy to see that 
this partition is equitable. The quotient matrix of $A(\Gamma_{r, s})$ with respect to this partition is 
\[
\begin{pmatrix}
0 & 0 & 0 & -1 & 1 & 0 \\
0 & 0 & 0 & 1 & 0 & s-2 \\
0 & 0 & 0 & 0 & 1 & s-2 \\
-1 & 1 & 0 & 0 & 0 & 0 \\
1 & 0 & r-2 & 0 & 0 & 0 \\
0 & 1 & r-2 & 0 & 0 & 0 
\end{pmatrix}.
\]
As is easily verified,  its characteristic polynomial is
\[
f(x)=x^2\left(x^4 -((r-1)(s-1)+2)x^2+(2r-3)(2s-3)\right).
\]
By Lemma \ref{QM},  the four nonzero roots of $f(x)=0$, which  are 
\[
\pm\sqrt{ \frac{1}{2}\left((r-1)(s-1)+2\pm\sqrt{((r-1)(s-1)+2)^2-4(2r-3)(2s-3)}\right)},
\]
are eigenvalues of $\Gamma_{r, s}$. By the above proof, $0$ is an eigenvalue of $\Gamma_{r,s}$  with multiplicity at least $n-6$. So we find $n-2$ eigenvalues of  $\Gamma_{r, s}$. Denote by $m$  the number of edges of $\Gamma_{r, s}$.
The remaining two eigenvalues must be $0$ as  the sum of the square of the nonzero roots of $f(x)=0$ is $2(r-1)(s-1)+4=2m$, and it is easily seen that 
 the sum of the square of all eigenvalues is equal to  $2m$. Thus $\lambda_1(\Gamma_{r, s})$ is just the largrst root $f(r,s)$ of $f(x)=0$. This proves (i).

Next, we show (ii). Note that $f(s,t)>\sqrt{(r-1)(s-1)}$ is equivalent to 
\[
(r-1)(s-1)+2+\sqrt{((r-1)(s-1)+2)^2-4(2r-3)(2s-3)}>2(r-1)(s-1),
\]
i.e.,
\[
\sqrt{((r-1)(s-1)+2)^2-4(2r-3)(2s-3)}>2(r-1)(s-2)-(r-1)(s-1)-2,
\]
which follows from the fact that
\begin{align*}
	& \quad ((r-1)(s-1)+2)^2-4(2r-3)(2s-3)-(2(r-1)(s-2)-(r-1)(s-1)-2)^2\\
	& = 4\left((s-2)r^2-2(2s-3)r+5s-7\right)\\
	& \ge 4(9(s-2)-6(2s-3)+5s-7)\\
	& = 2s-7\\
	& > 0.
\end{align*}
This proves (ii).

Suppose that $r+s=n$ is fixed. 
For $3\le t\le s$ and $s\ge 4$, let
\[
g(t)=(t-1)(s-1)+2+\sqrt{((t-1)(s-1)+2)^2-4(2s-3)(2t-3)}. 
\]
Then 
\[
\lambda_1(\Gamma_{r, s})=\sqrt{\frac{1}{2}g(r)}=f(r, s),
\]
As
\begin{align*}
	g'(t)& =s-1+\frac{((t-1)(s-1)+2)(s-1)-4(2s-3)}{\sqrt{((t-1)(s-1)+2)^2-4(2s-3)(2t-3)}}\\
	& =s-1+\frac{(s-1)((t-1)(s-1)-6)+4}{\sqrt{((t-1)(s-1)+2)^2-4(2s-3)(2t-3)}}\\
	& >0, 
\end{align*}
$g(t)$ is strictly increasing for $3\le t\le s$, so $f(r,s)=\sqrt{\frac{1}{2}g(r)}$ is strictly increasing for $r\le s$. This is (iii).
\end{proof}

\section{Proof of Theorem \ref{N1}}

\begin{proof}[Proof of Theorem~\ref{N1}]

Let $\mathbf{\Gamma}_{r, s}$ be the class of unbalanced signed $r$ by $s$ bipartite graphs containing no $C_4^-$.
Let $\Gamma=(G,\sigma)$ be a graph in $\mathbf{\Gamma}_{r, s}$ that maximizes the spectral radius. 
As $G$ is bipartite, $\rho(\Gamma)=\lambda_1(\Gamma)$. 
Let $X$ and $Y$ be the bipartite sets of $G$ with $|X|=r$ and $|Y|=s$.

Assume that  $V(G)=\{v_1,\dots, v_{n}\}$  with $n=r+s$ and $C=v_1v_2\dots v_tv_1$ is a shortest  negative induced cycle in $\Gamma$ with $v_1\in X$.  As $G$ is bipartite and $\Gamma$ contains no $C_4^-$, $t$ is even and $t\ge 6$.

If $s=3$, then $t=6$, so $\Gamma$ is a negative $C_6$, which is switching isomorphic to $\Gamma_{3, 3}$. By a direct calculation, $\rho(\Gamma)=\sqrt{3}$.  So the result follows for $s=3$. 

Suppose in the following that $s\ge 4$.  
By Lemma \ref{L+},  $\Gamma$ is switching equivalent to
a signed graph $\Gamma^*=(G,\sigma^*)$  such that $\lambda_1(\Gamma^*)$ has a non-negative unit eigenvector $\mathbf{x}=(x_1, \dots, x_n)^\top$. 
By Lemma \ref{SE}, $\Gamma^*$ is also an unbalanced signed $r$ by $s$ bipartite graph  containing no $C_4^-$, and  $C$ is still a shortest negative induced cycle in $\Gamma^*$.  
By Lemmas \ref{same} and \ref{comp} (ii) and the fact that $\Gamma_{r, s}\in \mathbf{\Gamma}_{r, s}$,  
\[
\lambda_1(\Gamma^*)=\lambda_1(\Gamma)\ge \lambda_1\left(\Gamma_{r, s}\right)>\sqrt{(r-1)(s-2)}.
\]

\begin{Claim}\label{nz}
	If $v_av_b$ is a negative edge of $\Gamma^*$ outside $C$, then $x_a=x_b=0$.
\end{Claim}

\begin{proof}
	Suppose that $x_a\ne 0$ or $x_b\ne 0$. Then $\Gamma^*-v_av_b\in \mathbf{\Gamma}_{r, s}$ as $v_av_b$ lies outside  $C$. By Lemma \ref{add}, we have $\lambda_1(\Gamma^*-v_av_b)>\lambda_1(\Gamma^*)$, a contradiction.
\end{proof}

\begin{Claim}\label{x2-0}
Let $S=\{v_i: x_i=0, i\in [n]\}$. If  $|S|\ge 3$,  then there are three vertices $v_i,v_j,v_k$ in $S$ such that $|E(G-\{v_i,v_j,v_k\})|\le (r-1)(s-2)$.
\end{Claim}

\begin{proof} 
If $S\cap X\ne \emptyset$, then for any three vertices $v_i,v_j,v_k$ in $S$ with at least one in $X$, 
\[
|E(G-\{v_i,v_j,v_k\})|\le \max\{(r-1)(s-2),(r-2)(s-1),(r-3)s\}=(r-1)(s-2),
\]
as desired.
Suppose that  $S\cap X= \emptyset$. Then  there is no negative edge of $\Gamma^*$ outside $C$, as otherwise $S\cap X\ne\emptyset$ by Claim \ref{nz}, a contradiction. 
As $C$ has both negative and positive edges, there are two edges $v_\ell v_{\ell+1}$ and $v_{\ell+1}v_{\ell+2}$ in $C$ such that they have different signature, where $\ell=1,3,\dots,t-1$ and $v_{\ell+2}=v_1$ if $\ell=t-1$.  If $v_\ell$ and $v_{\ell+2}$ have some neighbor $w$ different from $v_{\ell+1}$, then $v_\ell v_{\ell+1}v_{\ell+2}wv_\ell$ is a negative $C_4$, a contradiction. So $v_\ell$ and $v_{\ell+2}$ have no common neighbors outside $C$. Let $v_i, v_j, v_k$ be any three vertices in $S$.
If $v_{\ell+1}\notin \{v_i,v_j,v_k\}$, then $|N_{G-\{v_i,v_j,v_k\}}(v_{\ell})|+|N_{G-\{v_i,v_j,v_k\}}(v_{\ell+2})|\le |Y|-3<s-2$, and otherwise
$|N_{G-\{v_i,v_j,v_k\}}(v_{\ell})|+|N_{G-\{v_i,v_j,v_k\}}(v_{\ell+2})|\le |Y|-3+1=s-3+1=s-2$. In either case, $|N_{G-\{v_i,v_j,v_k\}}(v_{\ell})|+|N_{G-\{v_i,v_j,v_k\}}(v_{\ell+2})|\le s-2$.
So 
\begin{align*}
& \quad|E(G-\{v_i,v_j,v_k\})|\\ 
& =|N_{G-\{v_i,v_j,v_k\}}(v_{\ell})|+|N_{G-\{v_i,v_j,v_k\}}(v_{\ell+2})|+\sum_{w\in X\setminus\{v_{\ell},v_{\ell+2}\}}|N_{G-\{v_i,v_j,v_k\}}(w)|\\
&\le s-2+\sum_{w\in X\setminus\{v_{\ell},v_{\ell+2}\}}(s-3)\\
& =s-2+(r-2)(s-3)\\
& < (r-1)(s-2),
\end{align*}
as desired. 
\end{proof} 

\begin{Claim}\label{x2}
$\mathbf{x}$ contains at most two zero entries.
\end{Claim}	

\begin{proof} Suppose that $\mathbf{x}$ contains at least three zero entries. By Claim \ref{x2-0}, there are three vertices $v_i,v_j,v_k$ in $S$ such that $|E(G-\{v_i,v_j,v_k\})|\le (r-1)(s-2)$.  So 
\[
\lambda_1(G-\{v_i,v_j,v_k\})\le \sqrt{|E(G-\{v_i,v_j,v_k\})|}\le \sqrt{(r-1)(s-2)}.
\]
Let $\mathbf{y}$ be a vector obtained from $\mathbf{x}$ by deleting three entries $x_i$, $x_j$ and $x_k$.
Then by Rayleigh's principle and  Lemma \ref{under}, 
\begin{align*}
\lambda_1(\Gamma^*)& =\mathbf{x}^\top A(\Gamma^*) \mathbf{x}=\mathbf{y}^\top A(\Gamma^*-\{v_i,v_j,v_k\})\mathbf{y}\\
& \le \lambda_1(\Gamma^*-\{v_i,v_j,v_k\})\le \lambda_1(G-\{v_i,v_j,v_k\})\\
& \le \sqrt{(r-1)(s-2)},
\end{align*}
a contradiction.
\end{proof}

\begin{Claim}\label{zjia}
There is at most one negative edge outside $C$ in $\Gamma^*$.
\end{Claim}

\begin{proof}
It follows from  Claims \ref{x2} and \ref{nz}.
\end{proof}

\begin{Claim}\label{con}
$G$ is connected.
\end{Claim}

\begin{proof}
Suppose that $G$ is not connected. Assume that $v_1v_2$ is a negative edge of $C$ with $x_1\ge x_2$. There is a connected component $H$ containing $v_1v_2$. 
Assume that $v_n$ is a vertex in a component different from $H$  and  $v_1$ and $v_n$ lie in different partite sets of $G$. Then $\Gamma^*+v_1v_n$ contains no $C_4^-$ as $v_1v_n$ is a cut edge, and it is unbalanced as $C$ is also a negative cycle of $\Gamma^*+v_1v_n$.
So $\Gamma^*+v_1v_n\in \mathbf{\Gamma}_{r, s}$. By Claim \ref{x2}, we have $x_1\ne 0$ or $x_n\ne 0$.
By Lemma \ref{add}, we have  $\lambda_1(\Gamma^*+v_1v_n)>\lambda_1(\Gamma^*)$, a contradiction.
\end{proof}

\begin{Claim}\label{nout}
If there is a negative edge outside $C$ in $\Gamma^*$, then it does not lie on any $C_4$.
\end{Claim}

\begin{proof}
It is trivial that a negative edge outside $C$ in $\Gamma^*$  does not lie on any  negative $C_4$.

Suppose that there is a negative edge $v_av_b$ in $\Gamma^*$ with $1\le a<b\le n$ outside $C$ and it lies on some positive $C_4$. By Claim \ref{zjia}, $v_av_b$ is the unique 
negative edge outside $C$ in $\Gamma^*$. Then  the positive $C_4$ must contain a negative edge in $C$, say $v_1v_2$. Whether the two cycles share one or two edges,  $\Gamma^*-v_1v_2$ has a negative induced cycle consisting of the edges of the two cycles but not on both. 
So $\Gamma^*-v_1v_2\in \mathbf{\Gamma}_{r, s}$.
Assume that $x_1\ge x_2$.
By Claims \ref{x2} and \ref{nz}, we have $x_a=x_b=0$, and $x_1>0$.
By Lemma \ref{add}, we have $\lambda_1(\Gamma^*-v_1v_2)>\lambda_1(\Gamma^*)$, a contradiction.
\end{proof}

\begin{Claim}\label{one}
There is exactly one negative edge on $C$ in $\Gamma^*$.
\end{Claim}

\begin{proof} Suppose the claim is false. 
Then there are at least three negative edges on $C$ in $\Gamma^*$. Assume $v_1v_2$ and $v_iv_{i+1}$ are negative edges, where $2\le i\le t$ with $v_{t+1}=v_1$.
Let $\Gamma'$ be the signed graph obtained from $\Gamma^*$ by reverseing the sign of $v_1v_2$ and $v_iv_{i+1}$. 

Suppose first that $\Gamma'$ contains some $C_4^-$, say $C'$, which must be a positive $C_4$ in $\Gamma^*$, and it contains $v_1v_2$ or $v_iv_{i+1}$. By Claim \ref{nout},  it contains no negative edges outside $C$. So it contains a  negative edge on $C$ different from  $v_1v_2$ and $v_iv_{i+1}$, implying that there is a chord on $C$, a contradiction. 
%
%
%
Suppose next that $C'$ contains exactly one of $v_1v_2$ or $v_iv_{i+1}$, say $v_1v_2$. By Claim \ref{nout} and the fact that there is no chords on $C$,  $C'$ contains no negative edge outside $C$ and contains a negative edge on $C$ different from $v_1v_2$, which must be $v_2v_3$ or $v_tv_1$, say $v_2v_3$. Let $E_1$ be the set of edges  on $C$ and $C'$ except $v_1v_2$ and $v_2v_3$. It induces  a negative cycle of $\Gamma^*$.
By Claim \ref{x2}, one of $x_1,x_2,x_3$ is nonzero.
Let $\dot\Gamma=\Gamma^*-v_1v_2$ if $x_1\ne0$ or $x_2\ne0$, and $\dot\Gamma=\Gamma^*-v_2v_3$ if $x_3\ne0$.
Then $\dot\Gamma\in \mathbf{\Gamma}_{r, s}$. By Lemma \ref{add}, we have  $\lambda_1(\dot\Gamma)>\lambda_1(\Gamma^*)$, a contradiction.
It follows that $\Gamma'$ contains no $C_4^-$.
Thus $\Gamma'\in \mathbf{\Gamma}_{r, s}$.

By Rayleigh's principle,
\[
0\ge \lambda_1(\Gamma')-\lambda_1(\Gamma^*)\ge\mathbf{x}^T(A(\Gamma')-A(\Gamma^*))\mathbf{x}
=4(x_1x_2+x_ix_{i+1})\ge0,
\]
so the above inequalities are equalities, implying that $x_1x_2+x_ix_{i+1}=0$, $\lambda_1(\Gamma')=\lambda_1(\Gamma^*)$, and $\mathbf{x}$ is an eigenvector associated with $\lambda_1(\Gamma')$.
If $i=2$ or $t$, say $i=2$, then
\[
2x_2=(\lambda_1(\Gamma')-\lambda_1(\Gamma^*))x_1=0
\]
and
\[
2(x_1+x_3)=(\lambda_1(\Gamma')-\lambda_1(\Gamma^*))x_2=0,
\]
so $x_1=x_2=x_3=0$, contradicting Claim \ref{x2}. Thus $i\ne2,t$. Since
\[
2x_1=(\lambda_1(\Gamma')-\lambda_1(\Gamma^*))x_2=0
\]
and
\[
2x_2=(\lambda_1(\Gamma')-\lambda_1(\Gamma^*))x_1=0, 
\]
we have $x_1=x_2=0$. As $x_1x_2+x_ix_{i+1}=0$, we have $x_i=0$ or $x_{i+1}=0$, contradicting Claim \ref{x2}.
\end{proof}

By Claim \ref{one}, there is exactly one negative edge on $C$, which is assumed to be $v_1v_2$.

\begin{Claim}\label{non}
There is no negative edge outside $C$.
\end{Claim}

\begin{proof} Suppose that this is not the case. By Claim \ref{zjia}, 
There is a unique negative edge,  say $v_av_b$ outside $C$. 
By Claim \ref{nz}, we have $x_a=x_b=0$.

Recall that $v_1v_2$ is the unique negative edge on $C$.

\noindent
{\bf Case 1.} $\{v_a,v_b\}\cap V(C)\ne\emptyset$.

Assume that $v_a\in V(C)$ with $2 \le a \le \frac{t}{2}+1$.

Suppose first that $a=2$. Then $v_bv_t\notin E(G)$ by Claim \ref{nout}.
Note that $v_b$ and $v_t$ are in different partite sets and $C$ remains to be a negative cycle in $\Gamma^*+v_bv_t$. If $\Gamma^*+v_bv_t$ contains a negative $C_4$, then it contains $v_bv_t$ and a negative edge, which must be $v_2v_b$. So $v_2$ and $v_t$ have a common neighbor different from $v_1$, say $u$, then $v_1v_2uv_tv_1$ is a negative $C_4$ in $\Gamma^*$, a contradiction.
Thus  $\Gamma^*+v_bv_t\in \mathbf{\Gamma}_{r, s}$.
By Claims \ref{x2} and the fact that $x_a=x_b=0$, we have $x_t\ne 0$ .
So by Lemma \ref{add}, we have $\lambda_1(\Gamma^*+v_bv_t)>\lambda_1(\Gamma^*)$, a contradiction.

Suppose next that $3\le a \le \frac{t}{2}+1$. By Claim \ref{x2}, we have $x_i>0$ for $i=1, 2$.

Suppose that $a$ is odd. 
Let $\Gamma'=\Gamma^*-v_av_b+v_1v_b$. Obviously, $C$ is also a negative cycle of $\Gamma'$.
If $\Gamma'$ contains some $C_4^-$, this $C_4$ contains $v_1v_2$ and $v_1v_b$, so  $v_2$ and $v_b$ have a common neighbor different from $v_1$, say $u_1$. 
Then $u_1\ne v_3$ which is clear if $a=3$, and as otherwise, $v_3\dots v_av_bv_3$ is a shorter negative cycle than $C$ in $\Gamma^*$ for $a\ge 5$, a contradiction.
So $v_2v_3\dots v_av_bu_1v_2$ is a negative cycle in $\Gamma^*$, which is a negative $C_4$ if $a=3$, and  a shorter negative cycle than $C$ if $a\ge 5$, a contradiction. Thus $\Gamma'\in \mathbf{\Gamma}_{r, s}$. By Rayleigh's principle, we have 
\[
0\ge \lambda_1(\Gamma')-\lambda_1(\Gamma^*)\ge\mathbf{x}^\top(A(\Gamma')-A(\Gamma^*))\mathbf{x}
=2x_b(x_1+x_a)=0,
\]
so $\lambda_1(\Gamma')=\lambda_1(\Gamma^*)$ and $\mathbf{x}$ is also an eigenvector associated with $\lambda_1(\Gamma')$. Recall that $x_1>0$. Then 
\[
0< x_1=x_1+x_a=\left(\lambda_1(\Gamma')-\lambda_1(\Gamma^*)\right)x_b=0,
\]
a contradiction.

Suppose that $a$ is even.
We claim that $v_1$ and $v_b$ have no common neighbors. Otherwise, let $u_2\in N_{\Gamma^*}(v_1)\cap N_{\Gamma^*}(v_b)$. If $u_2=v_t$, then $v_bv_a\dots v_tv_b$
is a shorter negative cycle than $C$ in $\Gamma^*$, a contradiction. So $u_2\ne v_t$.
If $a\ge 6$, then  $v_bv_a\dots v_tv_1u_2v_b$  is a shorter negative cycle than $C$ in $\Gamma^*$, also a contradiction. If $a=4$, then $\Gamma^*-v_1v_2\in \mathbf{\Gamma}_{r, s}$ as $v_4\dots v_tv_1u_2v_bv_4$ is a negative cycle in $\Gamma^*-v_1v_2$, and we have by Lemma \ref{add} that $\lambda_1(\Gamma^*-v_1v_2)>\lambda_1(\Gamma^*)$, a contradiction.
It follows that $v_1$ and $v_b$ have no common neighbors, as claimed.
Let $\Gamma''=\Gamma^*-v_av_b+v_2v_b$, which is unbalanced as $C$ is also a negative cycle of $\Gamma''$.
Note that  $\Gamma''$ contains no $C_4^-$ as $v_1$ and $v_b$ have no common neighbors different from $v_2$.
So $\Gamma''\in \mathbf{\Gamma}_{r, s}$.  By Rayleigh's principle, we have 
\[
0\ge \lambda_1(\Gamma'')-\lambda_1(\Gamma^*)\ge2x_b(x_2+x_a)=0, 
\]
so $\lambda_1(\Gamma'')=\lambda_1(\Gamma^*)$ and $\mathbf{x}$ is also a eigenvector associated with $\lambda_1(\Gamma'')$. Recall that $x_2>0$. Then 
\[
0<x_2=x_2+x_a=\left(\lambda_1(\Gamma'')-\lambda_1(\Gamma^*)\right)x_b=0,
\]
a contradiction.

\noindent
{\bf Case 2.} $\{v_a,v_b\}\cap V(C)=\emptyset$.

By Claims \ref{x2}, $x_i>0$ for $i=1,\dots, t$

Assume that $v_1$ and $v_b$ are in the same partite set of $G$.

We claim that  $v_1$ is not adjacent to $v_a$. Suppose that this is not true.  Then $v_b$ is not adjacent to $v_2$ by Claim \ref{nout}. Note that $\Gamma^*+v_2v_b$ contains no $C_4^-$ as otherwise $v_1$ and $v_b$  have a common neighbor different from $v_a$, say $u'$, 
or $v_2$ and $v_a$ have a common neighbor different from  $v_b$, say $u''$, so $v_1v_av_bu'v_1$ or $v_1v_2u''v_av_1$ is a negative $C_4$ in $\Gamma^*$, a contradiction.
So $\Gamma^*+v_2v_b\in \mathbf{\Gamma}_{r, s}$. 
By Lemma \ref{add} in view of $x_2>0$,  $\lambda_1(\Gamma^*+v_2v_b)>\lambda_1(\Gamma^*)$, a contradiction. 
So $v_1$ is not adjacent to $v_a$, as claimed.

Similarly as above by considering $\Gamma^*+v_1v_a$, $v_2$ is not adjacent to $v_b$.

Suppose that $v_1$ and $v_b$ ($v_2$ and $v_a$, respectively) have a common neighbor, say $w$ ($z$, respectively).
Then $\Gamma^*-v_1v_2$ is unbalanced as it has a negative cycle $C'$, where 
\[
C'=\begin{cases}v_2\dots v_tv_1wv_bv_azv_2 & \text{if $w\ne v_t$ and $z\ne v_3$},\\ v_2\dots v_tv_bv_azv_2 &  \text{if $w=v_t$ and $z\ne v_3$},\\
v_3\dots v_tv_1wv_bv_av_3 & \text{if $w\ne v_t$ and $z=v_3$},\\ 
v_3\dots v_tv_bv_av_3 & \text{if $w=v_t$ and $z=v_3$}.
\end{cases}
\]
So $\Gamma^*-v_1v_2\in \mathbf{\Gamma}_{r, s}$.
By Lemma \ref{add}, we have $\lambda_1(\Gamma^*-v_1v_2)>\lambda_1(\Gamma^*)$, a contradiction.
This shows that   $v_1$ and $v_b$ or $v_2$ and $v_a$, say $v_1$ and $v_b$ have no common neighbors.
Then it can be easily checked that $\Gamma^*-v_av_b+v_2v_b\in \mathbf{\Gamma}_{r, s}$.
By Rayleigh's principle, we have
\[
0\ge \lambda_1\left(\Gamma^*-v_av_b+v_2v_b\right)-\lambda_1\left(\Gamma^*\right)\ge 2x_b\left(x_2+x_a\right)=0,
\]
implying that $\lambda_1\left(\Gamma^*-v_av_b+v_2v_b\right)=\lambda_1\left(\Gamma^*\right)$ and $\mathbf{x}$ is also an eigenvector associated with $\lambda_1\left(\Gamma^*-v_av_b+v_2v_b\right)$.
But then
\[
0<x_2=x_2+x_a=\left(\lambda_1\left(\Gamma^*-v_av_b+v_2v_b\right)-\lambda_1\left(\Gamma^*\right)\right)x_b=0,
\]
a contradiction.
\end{proof}

\begin{Claim}\label{xi}
$x_i>0$ for $i=3,\dots, n$.
\end{Claim}

\begin{proof}
Suppose that  $x_i=0$ for some  $i=3,\dots, n$.
By Claim \ref{non}, $v_1v_2$ is the unique negative of $\Gamma^*$.
Then
\[
\sum_{v_k\in N_{\Gamma^*}(v_i)}x_k=\sum_{v_k\in N_{\Gamma^*}(v_i)}\sigma^*(v_iv_k)x_k=\lambda_1(\Gamma^*)x_i=0,
\]
so $x_k=x_i=0$ for every  $v_k\in N_{\Gamma^*}(v_i)$. Thus $\mathbf{x}$ have at least  $|N_{\Gamma^*}(v_i)|+1$ zero entries. By  Claim \ref{x2}, $|N_{\Gamma^*}(v_i)|\le 1$. 
By  Claim \ref{con},
$|N_{\Gamma^*}(v_i)|=1$. 
 So $v_i$ lies outside $C$.
By Claim \ref{x2} again, we have $x_j>0$ for any $j\in \{1,\dots, n\}\setminus\{i, s\}$.
Assume that $x_1\ge x_2$. Then $s\ne1$ and $x_1>0$.
Let
\[
\Gamma'=\begin{cases}
\Gamma^*+v_5v_i & \mbox{if $s=3$, }\\
\Gamma^*+v_2v_i & \mbox{if $s=4$, }\\
\Gamma^*+v_4v_i & \mbox{if $v_1$ and $v_s$ are in different partition of $G$, and $s\ne 4$ }\\
\Gamma^*+v_3v_i & \mbox{otherwise}.
\end{cases}
\]
Clearly, $\Gamma'\in \mathbf{\Gamma}_{r, s}$.
By Lemma \ref{add}, we have $\lambda_1(\Gamma')>\lambda_1(\Gamma^*)$, a contradiction.
\end{proof}

Recall that $t$ is even.
If  $t\ge8$,
then $\Gamma^*+v_3v_{t-2}\in \mathbf{\Gamma}_{r, s}$, and we have by Lemmas \ref{add} and \ref{xi} that $\rho(\Gamma^*+v_3v_{t-2})>\rho(\Gamma^*)$, a contradiction. It follows that $t=6$,
so $C=v_1\dots v_6v_1$.

As $\Gamma^*$ contains no  $C_4^-$, $N_{\Gamma^*}(v_2)\cap N_{\Gamma^*}(v_6)=\{v_1\}$ and $N_{\Gamma^*}(v_1)\cap N_{\Gamma^*}(v_3)=\{v_2\}$.
Let $X_1=N_{\Gamma^*}(v_2)\setminus\{v_1, v_3\}$, $X_2=N_{\Gamma^*}(v_6)\setminus\{v_1, v_5\}$, $Y_1=N_{\Gamma^*}(v_1)\setminus\{v_2, v_6\}$ and $Y_2=N_{\Gamma^*}(v_3)\setminus\{v_2, v_4\}$.  Evidently, $X_1\cap X_2=\emptyset$ and $Y_1\cap Y_2=\emptyset$.

For any $u\in X_1$ and $w\in Y_1$, $u$ is not adjacent to $w$ as otherwise $v_1wuv_2v_1$ is a negative $C_4$, a contradiction.
Then by Claim \ref{xi} and Lemma \ref{add}, we have $X_1\cup X_2\subset N_{\Gamma^*}(v_4)$, $Y_1\cup Y_2\subset N_{\Gamma^*}(v_5)$, and for any $u\in X_1$, $v\in X_2$, $w\in Y_1$ and $z\in Y_2$, $u$ is adjacent to $z$, $v$ is adjacent to $w$ and $z$.

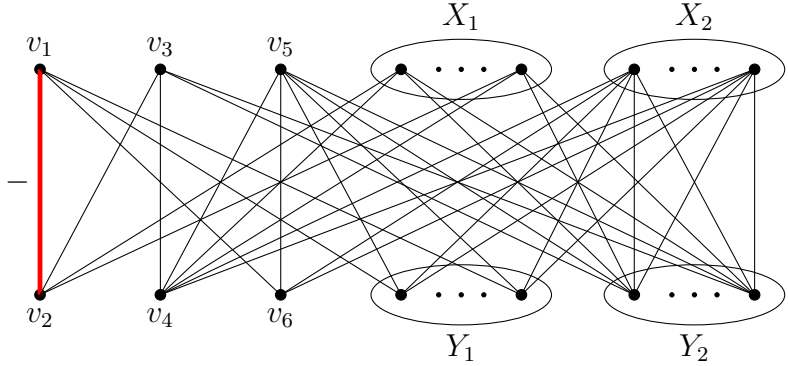
\begin{figure}[h]
	\centering
	\begin{tikzpicture}
		\filldraw [black] (1,-0.5) circle (2pt);
		\filldraw [black] (2.6,-0.5) circle (2pt);
		\filldraw [black] (4.2,-0.5) circle (2pt);
		\filldraw [black] (5.8,-0.5) circle (2pt);
		\filldraw [black] (6.3,-0.5) circle (0.8pt);
		\filldraw [black] (6.6,-0.5) circle (0.8pt);
		\filldraw [black] (6.9,-0.5) circle (0.8pt);
		\filldraw [black] (7.4,-0.5) circle (2pt);
		\filldraw [black] (8.9,-0.5) circle (2pt);
		\filldraw [black] (9.4,-0.5) circle (0.8pt);
		\filldraw [black] (9.7,-0.5) circle (0.8pt);
		\filldraw [black] (10.0,-0.5) circle (0.8pt);
		\filldraw [black] (10.5,-0.5) circle (2pt);
		
		\filldraw [black] (1,-3.5) circle (2pt);
		\filldraw [black] (2.6,-3.5) circle (2pt);
		\filldraw [black] (4.2,-3.5) circle (2pt);
		\filldraw [black] (5.8,-3.5) circle (2pt);
		\filldraw [black] (6.3,-3.5) circle (0.8pt);
		\filldraw [black] (6.6,-3.5) circle (0.8pt);
		\filldraw [black] (6.9,-3.5) circle (0.8pt);
		\filldraw [black] (7.4,-3.5) circle (2pt);
		\filldraw [black] (8.9,-3.5) circle (2pt);
		\filldraw [black] (9.4,-3.5) circle (0.8pt);
		\filldraw [black] (9.7,-3.5) circle (0.8pt);
		\filldraw [black] (10.0,-3.5) circle (0.8pt);
		\filldraw [black] (10.5,-3.5) circle (2pt);
		\draw  [red, line width=1.8pt] (1,-0.5)--(1,-3.5);
		\node at (0.7,-2) {$-$};
		\draw  [black](1,-3.5)--(2.6,-0.5)--(2.6,-3.5)--(4.2,-0.5)--(4.2,-3.5)--(1,-0.5);
		\draw  [black](1,-0.5)--(5.8,-3.5);
		\draw  [black](1,-0.5)--(7.4,-3.5);
		\draw  [black](2.6,-0.5)--(8.9,-3.5);
		\draw  [black](2.6,-0.5)--(10.5,-3.5);
		\draw  [black](4.2,-0.5)--(5.8,-3.5);
		\draw  [black](4.2,-0.5)--(7.4,-3.5);
		\draw  [black](4.2,-0.5)--(8.9,-3.5);
		\draw  [black](4.2,-0.5)--(10.5,-3.5);
		\draw  [black](1,-3.5)--(5.8,-0.5);
		\draw  [black](1,-3.5)--(7.4,-0.5);
		\draw  [black](4.2,-3.5)--(8.9,-0.5);
		\draw  [black](4.2,-3.5)--(10.5,-0.5);
		\draw  [black](2.6,-3.5)--(5.8,-0.5);
		\draw  [black](2.6,-3.5)--(7.4,-0.5);
		\draw  [black](2.6,-3.5)--(8.9,-0.5);
		\draw  [black](2.6,-3.5)--(10.5,-0.5);
		\draw (6.6,-0.5) ellipse (1.2 and 0.4);
		\draw (9.7,-0.5) ellipse (1.2 and 0.4);
		\draw (6.6,-3.5) ellipse (1.2 and 0.4);
		\draw (9.7,-3.5) ellipse (1.2 and 0.4);
		\draw  [black] (5.8, -0.5)--(8.9, -3.5);
		\draw  [black] (5.8, -0.5)--(10.5, -3.5);
		\draw  [black] (7.4, -0.5)--(8.9, -3.5);
		\draw  [black] (7.4, -0.5)--(10.5, -3.5);
		\draw  [black] (5.8, -3.5)--(8.9, -0.5);
		\draw  [black] (5.8, -3.5)--(10.5, -0.5);
		\draw  [black] (7.4, -3.5)--(8.9, -0.5);
		\draw  [black] (7.4, -3.5)--(10.5, -0.5);
		\draw  [black] (8.9, -0.5)--(8.9, -3.5);
		\draw  [black] (8.9, -0.5)--(10.5, -3.5);
		\draw  [black] (10.5, -0.5)--(8.9, -3.5);
		\draw  [black] (10.5, -0.5)--(10.5, -3.5);
		\node at (6.6, 0.2) {$X_1$};
		\node at (9.7, 0.2) {$X_2$};
		\node at (6.6, -4.2) {$Y_1$};
		\node at (9.7, -4.2) {$Y_2$};
		\node at (1, -0.2) {$v_1$};
		\node at (1, -3.8) {$v_2$};
		\node at (2.6, -0.2) {$v_3$};
		\node at (2.6, -3.8) {$v_4$};
		\node at (4.2, -0.2) {$v_5$};
		\node at (4.2, -3.8) {$v_6$};
	\end{tikzpicture}
	
	\caption{The structure of $\Gamma^*$.}
	\label{f3}
\end{figure}

Let $X$ and $Y$ be the bipartite sets of $G$ with $v_1\in X$. Then $X_1\cup X_2\cup\{v_1,v_3,v_5\}\subseteq X$.
Suppose that $v'\in X\setminus \left(X_1\cup X_2\cup\{v_1,v_3,v_5\}\right)$. Let $\Gamma'=\Gamma^*+v_2v'$ if $v'\notin N_\Gamma^*(v)$ for any $v\in Y_1$, and $\Gamma'=\Gamma^*+v_6v'$ otherwise. Clearly, $\Gamma'\in \mathbf{\Gamma}_{r, s}$. By Claim \ref{xi}, $x_{v'}>0$. So
by Lemma \ref{add} , we have $\lambda_1(\Gamma')>\lambda_1(\Gamma^*)$, a contradiction.
So $X=X_1\cup X_2\cup\{v_1,v_3,v_5\}$.
Similarly, we have $Y=Y_1\cup Y_2\cup\{v_2,v_4,v_6\}$.
One can see in Fig. \ref{f3}  the rough structure of $\Gamma^*$. 
Note that $X_1\cup X_2\cup Y_1\cup Y_2\ne\emptyset$ as $n\ge7$.

By symmetry, we have $x_3=x_k$ for any $v_k\in X_1$, $x_5=x_k$ for any $v_k\in X_2$, $x_4=x_k$ for any $v_k\in Y_2$ and $x_6=x_k$ for any $v_k\in Y_1$.

\begin{Claim}\label{l0l}
$|Y_1|\le|Y_2|$ and $|X_1|\le|X_2|$.
\end{Claim}

\begin{proof}
Note that 
\[
\lambda_1(\Gamma^*)x_2=x_3+\sum_{v_k\in X_1}x_k-x_1.
\]
Then 
\[
\left(|X_1|+1\right)x_3-x_1\ge0.
\]
We claim that $|Y_1|\le|Y_2|$. 
Suppose that $|Y_1|>|Y_2|$, implying that $|Y_1|\ge 1$. 
Let 
\[
\Gamma'=\Gamma^*-v_1v_6+\{v_6v: v\in X_1\cup\{v_3\}\}.
\]
It is easy to see that $\Gamma'\in \mathbf{\Gamma}_{r, s}$. 
By Rayleigh's principle, 
\[
0\ge \lambda_1(\Gamma')-\lambda_1(\Gamma^*)\ge 2x_6\left((|X_1|+1)x_3-x_1\right)\ge 0, 
\]
so $\lambda_1(\Gamma')=\lambda_1(\Gamma^*)$ and $\mathbf{x}$ is also an eigenvector associated with $\lambda_1(\Gamma')$. But then 
\[
0=\left(\lambda_1(\Gamma')-\lambda_1(\Gamma^*)\right)x_1=-x_6, 
\]	
i.e., $x_6=0$, contradicting  Claim \ref{xi}. Thus we indeed have $|Y_1|\le|Y_2|$.

Similarly, we have $|X_1|\le|X_2|$ by constructing $\Gamma''=\Gamma^*-v_2v_3+\{v_3v: v\in Y_1\cup\{v_6\}\}$ as  $\left(|Y_1|+1\right)x_6-x_2=\lambda_1(\Gamma^*)x_1\ge0$. 
\end{proof}

\begin{Claim}\label{ll}
	$x_3\ge x_1$ and $x_6\ge x_2$.
\end{Claim}

\begin{proof}	
Note that 
\begin{align*}
\lambda_1(\Gamma^*)x_2&=-x_1+x_3+x_5+\sum_{v_k\in X_1}x_k,\\
\lambda_1(\Gamma^*)x_6&=x_1+x_5+\sum_{v_k\in X_2}x_k,\\
\lambda_1(\Gamma^*)x_4&=x_3+x_5+\sum_{v_k\in X_1\cup X_2}x_k.
\end{align*}
Then 
\[
\lambda_1(\Gamma^*)(x_2+x_6)=\lambda_1(\Gamma^*)x_4
\] 
By Claim \ref{xi}, we have $x_2=x_4-x_6$. 
Then 
\begin{align*}
\lambda_1(\Gamma^*)(x_3-x_1)& =2x_2+\left(|Y_2|+1\right)x_4-\left(|Y_1|+1\right)x_6\\
&=2(x_4-x_6)+2\left(|Y_2|+1\right)x_4-\left(|Y_1|+1\right)x_6\\
& =\left(|Y_2|+3\right)x_4-\left(|Y_1|+3\right)x_6.
\end{align*}
By Claim \ref{l0l}, $|Y_2|\ge |Y_1|$, so 
\[
\lambda_1(\Gamma^*)(x_3-x_1) \ge \left(|Y_1|+3\right)(x_4-x_6)=\left(|Y_1|+3\right)x_2\ge 0.
\]
Thus  $x_3\ge x_1$. 

Similarly, we have  $x_1+x_3=x_5$, and by Claim \ref{l0l}, $|X_1|\le|X_2|$, so
$\lambda_1(\Gamma^*)(x_6-x_2)\ge 0$,  implying that $x_6\ge x_2$.  ????
\end{proof}

\begin{Claim}\label{fin}
	$|X_1|=|Y_1|=0$. 
\end{Claim}

\begin{proof}
Suppose first that $|X_1|$, $|Y_1|\ge 1$. 
Let 
\begin{align*}
\widetilde\Gamma& =\Gamma^*-\{v_1v: v\in Y_1\}-\{v_2v: v\in X_1\}+\{v_3v: v\in Y_1\}\\
& \quad+\{v_6v: v\in X_1\}+\{uv: u\in X_1, v\in Y_1\}.
\end{align*}
It is easy to see that $\widetilde\Gamma$ is switching isomorphic  to $\Gamma_{r, s}\in \mathbf{\Gamma}_{r, s}$ (displayed in Fig.~\ref{f2}).
By Rayleigh's principle, and Claims \ref{ll} and \ref{xi}, we have 
\begin{align*}
0\ge \lambda_1(\widetilde\Gamma)-\lambda_1(\Gamma^*)& \ge 2(x_3-x_1)\sum_{v_k\in Y_1}x_k+2(x_6-x_2)\sum_{v_k\in X_1}x_k+2\sum_{v_k\in X_1}\sum_{v_ell\in Y_1}x_kx_\ell\\
& = 2\left(|Y_1|(x_3-x_1)x_6+|X_1|(x_6-x_2)x_3+|X_1||Y_1|x_3x_6\right)\\
& \ge 2|X_1||Y_1|x_3x_6\\
& >0,
\end{align*}
a constraction. 
So $|X_1|=0$ or $|Y_1|=0$.

Suppose next that $|X_1|=0$ and $|Y_1|\ne 0$, or $|X_1|\ne 0$ and  $|Y_1|=0$, say  $|X_1|=0$ and $|Y_1|\ne 0$. 
Let
\[
\dot{\Gamma}=\Gamma^*-\{v_1v: v\in Y_1\}+\{v_3v: v\in Y_1\}.
\]
We can see that $\dot{\Gamma}$ is switching isomorphic to $\Gamma_{r, s}$.
By Rayleigh's principle, and Claim \ref{ll}, we have  
\[
0\ge \lambda_1(\dot{\Gamma})-\lambda_1(\Gamma^*)\ge 2|Y_1|(x_3-x_1)x_6\ge 0, 
\]
so $\lambda_1(\dot\Gamma)=\lambda_1(\Gamma^*)$ and $\mathbf{x}$ is also an eigenvector associated with $\lambda_1(\dot\Gamma)$.
Note that 
\[
0=\left(\lambda_1(\dot{\Gamma})-\lambda_1(\Gamma^*)\right)x_1=-|Y_1|x_6, 
\]
implying that $x_6=0$, contradicting  Claim \ref{xi}. 
%
%
\end{proof}

By Claim \ref{fin},  $\Gamma^*$ is switching isomorphic to $\Gamma_{r, s}$, so  $\Gamma$ is switching isomorphic to $\Gamma_{r, s}$, and  we have by Lemmas \ref{same}  and \ref{comp} (i) that 
\begin{align*}
\lambda_1(\Gamma)&=\lambda_1(\Gamma^*)=\lambda_1(\Gamma_{r, s})\\
&= \sqrt{ \frac{1}{2}\left((r-1)(s-1)+2+\sqrt{((r-1)(s-1)+2)^2-4(2r-3)(2s-3)}\right)}\,, 
\end{align*}
as desired.
\end{proof}

\section{Proof of Theorem \ref{N2}}

\begin{proof}[Proof of Theorem~\ref{N2}]
Let $\Gamma=(G,\sigma)$ be an unbalanced signed bipartite graph of order $n$  containing no $C_4^-$ that maximizes the spectral radius. As $G$ is bipartite, we have  $\rho(\Gamma)=\lambda_1(\Gamma)$. 
Denote by $r$ and $n-r$ the partite sizes of $G$ with $r\le n-r$.
By Theorem \ref{N1},  
$\Gamma$ is switching isomorphic to $\Gamma_{r,n-r}$, and 
\[
\rho(\Gamma)=f(r,s):=\sqrt{ \frac{1}{2}\left((r-1)(s-1)+2+\sqrt{((r-1)(s-1)+2)^2-4(2r-3)(2s-3)}\right)}.
\]
By Lemma \ref{comp} (iii), we have $r=\left\lfloor\frac{n}{2}\right\rfloor$. The result follows.
\end{proof}

\bigskip

\noindent {\bf Acknowledgements.}
This work was supported by the
National Natural Science Foundation of China (No.~12571364) and 
the Guangdong Basic and Applied Basic Research Foundation 
(No.~2024A1515010493).

\bigskip

\noindent {\bf Declaration of competing interest}

\noindent
There is no competing interest.

\bigskip

\noindent {\bf Data Availability}

\noindent
There is no data associated with this paper.

\end{document}